\documentclass[12pt,reqno]{amsart}
\usepackage[utf8]{inputenc}
\usepackage{amsmath}
\usepackage{amscd}
\usepackage{graphics}
\usepackage{latexsym}
\usepackage{color}
\usepackage{verbatim}
\usepackage{extarrows}
\usepackage{latexsym}
\usepackage{enumerate}
\usepackage{arydshln}
\usepackage{soul}
\usepackage{tikz}
\usepackage{tikz-cd}
\usepackage{ulem} 
\usepackage{MnSymbol}

\usepackage{dsfont}
\usepackage{marvosym}
\usepackage{mathrsfs}

\definecolor{rose}{rgb}{0.82, 0.1, 0.26}
 	\definecolor{blu}{rgb}{0.36, 0.54, 0.66}
 	\definecolor{mor}{rgb}{0.55, 0.0, 0.55}

   \usepackage{hyperref}
 \hypersetup{
    colorlinks=true,
    linkcolor=rose, 
    filecolor=blue     }

\usepackage[alphabetic]{amsrefs}

    \hypersetup{ citecolor=mor} 
  \hypersetup{ filecolor=blue}

\usepackage[margin=1in]{geometry}
\usepackage[all]{xy}
\theoremstyle{plain}
\newtheorem{theorem}{Theorem}[section]

\newtheorem{cor}[theorem]{Corollary}
\newtheorem{lem}[theorem]{Lemma}
\newtheorem{prop}[theorem]{Proposition}

\theoremstyle{definition}
\newtheorem{definition}[theorem]{Definition}

\newtheorem{rmk}[theorem]{Remark}

\newtheorem*{thm1}{Theorem A}

\theoremstyle{remark}

\newcommand{\PP}{\mathds{P}}
\newcommand{\PPP}{\PP^{2}}

\newcommand{\ii}{\mathcal{I}}
\newcommand{\oo}{\mathcal{O}}

\usepackage{graphicx}
\usepackage{subcaption}
\usepackage{wrapfig}

\title{Extremal divisors in the Hilbert scheme of points on $\PPP$ are preserved under residuality }
\author{ Montserrat Vite}
\address{Centro de Ciencias Matemáticas, Morelia, Michoacan.}
\email{montserrat@matmor.unam.mx} 
\date{Enero 2025}
\begin{document}
\maketitle
\begin{abstract}
    Let $n=\frac{r(r+1)}{2}$ or $n=r(r+1)$. We prove that the property of being extremal is preserved under residuality on the Hilbert scheme of $n$ points in the plane.
\end{abstract}

The notion of linkage (liaison) of curves comes from M. Noether's thesis: Given a smooth curve in $\PP^{3}$, consider two hypersurfaces $F,G$ of degree $m,n$ with no common component and containing $C$. The complete intersection $F\cap G$ contains $C$ and a possible curve $D$. In this case we say that $C$ and $D$ are linked. The linkage of curves is an important tool in the study and classification of space curves and their Hilbert schemes (cf. \cite{mdp2}).  Many important questions have been answered using linkage of curves, and nowadays many questions still remain open on the matter (cf. \cite{mignag}). With respect to linkage of curves a very natural question is the following: 
\begin{center}
    \textit{What properties are preserved under linkage (liaison)?}
\end{center}

A first answer to this question is the property of being \textit{arithmetically Cohen-Macaulay (ACM for short)} that is, an ACM curve is ACM after linkage (liaison). (cf. \cite{apery} and \cite{gaeta}). On the other hand, the deficiency modules are preserved under (even) liaison (cf. \cite{mig2}). We are curious to find properties preserved by liaison related to the birational geometry of the varieties, such as the property of lying in an extremal divisor.
 
To be more specific, in \cite{vite} we proved that the closure of ACM curves $\overline{\mathscr{C}_{r}}$ in the Hilbert scheme of curves of degree $d_{r}=\frac{r(r+1)}{2}$ and genus $g_{r}=\frac{r(r+1)(2r-5)}{6}+1$ is irreducible and has two irreducible divisors, each parametrizes curves in the linkage class of two disjoint lines. All these divisors are related by liaison. In the case $r=3$ (degree six and genus three) we verify that these divisors are extremal in the effective cone of $\overline{\mathscr{C}_{r}}$.

 We ask if for the case $r\geq 4$ these divisors will also be extremal. Unfortunately, we could not solve that question, but it is natural to ask the same question in a smaller dimension. This situation motivated us to write these notes where we answer an analogous question in the case of the Hilbert scheme of $n$ points, $\PP^{2[n]}$.

The main advantage to consider in this case is because the birational geometry of the Hilbert scheme of $n$ points on $\PPP$ is very well known. For example, in \cite{abch} resolve the birational geometry of the Hilbert scheme of $n$ points for $n\leq 9$, they discuss the stable base locus decomposition of the effective cone and the corresponding birational models. On the other hand, in \cite{hui} compute the effective cone Eff$(\PP^{2[n]})$ for all values of $n$.

In this context (points in $\PPP$) the definition of liaison is known as residual points:

\begin{definition}
    Let $Z$ a subset of $n$ points in $\PPP$. We say that a subset $Z^{\prime}$ of $m$ points in $\PPP$ is \textit{residual} to $Z$ if there exist two curves $X$ and $Y$ on $\PPP$ such that
    $$X\cap Y =Z\cup Z^{\prime}. $$
\end{definition}

If $Z\in{\PP^{2[n]}}$, with $n$ a triangular number ($n=\frac{r(r+1)}{2}$ for some $r\in{\mathds{N}^{*}}$) and $Z^{\prime}\in{\PP^{2[m]}}$ is residual to $Z$ by the complete intersection of curves of degree $r+1$ then $m=\frac{(r+1)(r+2)}{2}$. Furthermore, we have the following theorem:
\begin{thm1}[Corollary \ref{corresid}]
Let $Z$ and $Z^{\prime}$ as before, then $Z$ lies in an extremal divisor if and only if $Z^{\prime}$ lies in the correspondent extremal divisor.
\end{thm1}

An analogous result is obtained in the case where $n$ is a tangential number ($n=2r(r+1)$ for some $r\in{\mathds{N}^{*}}$).

\section{Triangular points}

Let $Z_{r}$ be a set of $d_{r}:=\frac{r(r+1)}{2}$ general points in the complex projective plane $\mathds{P}^{2}$, this scheme have the following minimal free resolution:
\[ \xymatrix{ 0\ar[r]&\mathcal{O}_{\mathds{P}^{2}}(-(r+1))^{r}\ar[r]^{M_{r}}& \mathcal{O}_{\mathds{P}^{2}}(-r)^{r+1}\ar[r] &  \mathscr{I}_{Z_{r}} \ar[r] &0 }\]
If $\mathds{P}^{2[d_{r}]}$ denote the Hilbert scheme of $d_{r}$ points in the plane. A very well known divisor in this scheme is:
\[D_{r-1}:=\{ Z\in{\mathds{P}^{2[d_{r}]}}|h^{0}(\mathscr{I}_{Z}(r-1))=1\}\]

\begin{rmk} \label{rmk1.1}
 In \cite{fogarty}, proved that the Picard group of the Hilbert scheme of points $\PP^{2[n]}$ is the free abelian group generated by $\oo_{\PP^{2[n]}}(H)$ and $\oo_{\PP^{2[n]}}(\frac{B}{2})$ and the Neron-Severi space $N^{1}(\PP^{2[n]}) = Pic(\PP^{2[n]}) \otimes \mathds{Q}$ and is spanned by the divisor classes $H$ and $B$. Where $H$ is the class of the locus of subschemes $Z$ in $\PP^{2[n]}$ whose supports intersect a fixed line $\ell \subseteq \PPP$ and $B$ is the class of the locus of non-reduced schemes. On the other hand, the effective cone of $\PP^{2[d_{r}]}$ is generated by $B$ and the class of $D_{r-1}$ by \cite{hui}.
 \end{rmk}
In the case $r=3$ we have that $D_{2}$ is the subscheme of six points that lies in a conic. Let $Z\in{D_{2}}$ be a generic element, since its ideal is generated by a conic and a cubic, we can consider two smooth curves $X$ and $Y$ of degree four that contain the points in $Z$ and take its residual, that means, $Z^{\prime}:=X\cap Y-Z$, since the degree of the intersection $X\cap Y$ is 16 and the degree of $Z$ is 6 then $Z^{\prime}\in{\mathds{P}^{2[10]}=\mathds{P}^{2[d_{4}]}}$. By Riemann-Roch we have that
$$h^{0}(\mathscr{I}_{Z}(2))= h^{0}(\mathscr{I}_{Z^{\prime}}(3))$$

By hypothesis $h^{0}(\mathscr{I}_{Z}(2))=1$ thus $h^{0}(\mathscr{I}_{Z^{\prime}}(3))=1$ and by definition implies that $Z^{\prime}\in{D_{3}}$. 

\begin{definition}
    Given a family $\mathscr{A}$ in $\PP^{2[d_{r}]}$ we define a family $\mathscr{L}\mathscr{A}$ given by 
    \[\mathscr{L}\mathscr{A}:=\{ Z^{\prime}\in{\mathds{P}^{2[d_{r+1}]}}| \exists \quad Z\in{\mathscr{A}} \text{ such that }  Z^{\prime}=X_{r+1} \cap Y_{r+1}-Z\}\subseteq \PP^{2[d_{r+1}]}\]
    with $X_{r+1}$ and $Y_{r+1}$ curves of degree $r+1$
\end{definition}
    
The previous discussion proved that $\mathscr{L}D_{2}=D_{3}$. In general, we will prove $\mathscr{L}D_{r}=D_{r+1}$ for all $r\geq 2$. By the remark \ref{rmk1.1} this is exactly the affirmation of Theorem A. For this, we need the following Lemma.

\begin{lem} \label{lema1}
    If $Z_{r}\in{\mathds{P}^{2[d_{r}]}}$ and $Z_{r+1}\in{\mathds{P}^{2[d_{r+1}]}}$ are residuals via the complete intersection of two smooth curves of degree $r+1$ then:
\[h^{0}(\mathscr{I}_{Z_{r}}(r-1))=h^{0}(\mathscr{I}_{Z_{r+1}}(r))\]
\end{lem}
\begin{proof} Let $X,Y$ two smooth curves of degree $r+1$ such that $X\cap Y= Z_{r}\cup Z_{r+1}$, this is possible because the ideal of any element of $\PP^{2[d_{r}]}$ is generated by curves of degree $r$.

 Since the degree of $X$ is $r+1$, the genus of $X$ is $\frac{r(r-1)}{2}$, its canonical divisor have class $K_{X}=(r+1-3)H=(r-2)H$ and the class of $X\cap Y$ is $(r+1)H$ (where $H$ denote the hyperplane class of $X$). 

We consider to $Z_{r}$ as a divisor on $X$, for the above it has class $D=(r+1)H-Z_{r+1}$, by Riemann-Roch and the last equalities we have that:
\begin{align*}
h^{0}(\mathscr{I}_{Z_{r}}(r-1))&=h^{0}(\mathcal{O}_{X}((r-1)H-D))\\
&=h^{0}(\mathcal{O}_{X}(K_{X}+D-(r-1)H)) + grad((r-1)H-D)+1-g(X) \\
&=h^{0}(\mathcal{O}_{X}((r-2)H+D-(r-1)H)) + grad((r-1)H)-grad(D)+1-\frac{r(r-1)}{2} \\
&=h^{0}(\mathcal{O}_{X}(D-H)) + (r-1)(r+1)-\frac{r(r+1)}{2}+1-\frac{r(r-1)}{2} \\
&=h^{0}(\mathcal{O}_{X}((r+1)H-Z_{r+1}-H)) +0 \\
&=h^{0}(\mathcal{O}_{X}(rH-Z_{r+1}))  \\
&=h^{0}(\mathscr{I}_{Z_{r+1}}(r))
\end{align*}
\end{proof}

\begin{cor} \label{corresid}
For all $r\geqslant 3$ we have that $\mathscr{L} D_{r-1}=D_{r}$.
\end{cor}
\begin{proof}  If $Z_{r+1}$  is a generic element of $\mathscr{L}D_{r-1}$ by definition there exists an element $Z_{r}\in{D_{r-1}}$ residual to $Z_{r+1}$ via the complete intersection of two curves of degree $r+1$, by the Lemma \ref{lema1} we have that $1=h^{0}(\mathscr{I}_{Z_{r}}(r-1))=h^{0}(\mathscr{I}_{Z_{r+1}}(r))$ then $Z_{r+1}\in{D_{r}}$.

Let $Z_{r+1}$ a generic element of $D_{r}$, we know that its ideal is generated by an equation of degree $r$ and an equation of degree $r-1$, then we can consider two smooth surfaces $X$ and $Y$, of degree $r+1$ that contain $Z_{r+1}$ and do not have common components and take $Z_{r}=X\cap Y-Z_{r+1}$. Again by the Lemma \ref{lema1} and since $Z_{r+1}\in{D_{r}}$ we have that $1=h^{0}(\mathscr{I}_{Z_{r+1}}(r))=h^{0}(\mathscr{I}_{Z_{r}}(r-1))$ then $Z_{r}\in{D_{r-1}}$ thus $Z_{r+1}\in{\mathscr{L}D_{r-1}}$.\\

\end{proof}

Following the notation of \cite{vite} we have the next corollary:

\begin{cor}
    For $r\geq 3$. Given $C\in{\mathscr{L}^{r-3}\mathscr{C}^{h}\cup \mathscr{L}^{r-3}\mathscr{A} \subseteq \mathscr{H}_{r}}$ and $H$ a general hyperplane then $C\cap H\in{D_{r-1}}$.
\end{cor}
\begin{proof}
    By \cite{vite}*{Thm. 4.1} we have that when $r$ is an odd number then $\mathscr{L}^{r-3}\mathscr{C}^{h}=\mathscr{D}_{r-1}$ and when $r$ is an even number then $\mathscr{L}^{r-3}\mathscr{A}=\mathscr{D}_{r-1}$ where $\mathscr{D}_{r-1}$ is the set of curves contained in a surface of degree $r-1$. Thus, in both cases if $C\in{\mathscr{D}_{r-1}}$ is contained in $S$ a surface of degree $r-1$ then $S\cap H$ is a curve of degree $r-1$ containing $C\cap H$ and by definition this implies that $C\cap H\in{D_{r-1}}$.

    Suppose that $r$ is an even number and let $C\in{\mathscr{L}^{r-3}\mathscr{C}^{h}}$. Then we can consider two surfaces $S$ and $S^{\prime}$ of degree $r+1$ containing $C$ without common components. Let $C^{\prime}=S\cap S^{\prime}-C$, by definition $C^{\prime}\in{\mathscr{L}(\mathscr{L}^{r-3}\mathscr{C}^{h})=\mathscr{L}^{(r+1)-3}\mathscr{C}^{h}}$ and again by \cite{vite}*{Thm. 4.1} this implies that $C^{\prime}$ is contained in a surface of degree $r$ thus $Z^{\prime}=C^{\prime}\cap H$ is contained in a curve of degree $r$, but $Z=C\cap H =(S\cap H) \cap (S^{\prime}\cap H)-Z^{\prime}$ then $Z$ is residual to $Z^{\prime}$ by the complete intersection of the curves $S\cap H$ and $S^{\prime}\cap H$ of degree $r+1$, therefore by the Corollary \ref{corresid} we have that $Z\in{D_{r-1}}$.

    The case when $r$ is an odd number and $C\in{\mathscr{L}^{r-3}\mathscr{A}}$ is analogous.\\
\end{proof}

\section{Tangential points}

For $r\geq 1$, let $Z$ be a set of $d_{r}=2r(r+1)$ general points in the projective plane $\mathds{P}^{2}$, this scheme has minimal free resolution:
\begin{equation} \label{eqgentan}
0\to \mathcal{O}_{\mathds{P}^{2}}(-2(r+1))^{r}\to \mathcal{O}_{\mathds{P}^{2}}(-2r)^{r+1}\to \mathscr{I}_{Z}\to 0
\end{equation}

The bundle $T_{\PP^{2}}(2r-2)$ satisfies interpolation for $d_{r}$ general points, and the  Brill-Noether divisor $D_{r}:=D_{T_{\PP^{2}}(2r-2)}$ parametrize subsquemes $Z\in{\PP^{2[d_{r}]}}$ for which the multiplication morphism
$$ \mu_{2s}:H^{0}(\PP^{2},\mathscr{I}_{Z}(2s)) \otimes H^{0}(\PP^{2},\mathcal{O}_{\PP^{2}}(1)) \to H^{0}(\PP^{2},\mathscr{I}_{Z}(2s+1))$$
is not subjective.

\begin{rmk}
    By \cite{hui}, the effective cone of $\PP^{2[d_{r}]}$ is generated by $B$ (remark \ref{rmk1.1}) and the class of $D_{r}$.
\end{rmk}

\begin{rmk} \label{obs1}
    An alternative definition of the divisor $D_{r}$ is the complement of the open set of points $Z\in{\PP^{2[d_{r}]}}$ that have a resolution as (\ref{eqgentan}). The general element of $Z\in{D_{r}}$ admits an exact sequence of the form:
    \begin{equation} \label{restangext}
     \xymatrix{   0 \to \mathcal{O}_{\PP^{2}} (-2r-1) \oplus\mathcal{O}_{\PP^{2}}(-2r-2)^{r} \ar[r]^{\quad\varphi}& \mathcal{O}_{\PP^{2}}(-2r)^{r+1} \oplus \mathcal{O}_{\PP^{2}}(-2r-1) \ar[r]& \mathscr{I}_{Z} \to 0}
    \end{equation}
\end{rmk}

\begin{definition}
     Given a family $\mathscr{A}\subseteq \PP^{2[d_{r}]}$ we define:
    $$\mathscr{L}\mathscr{A}:=\{Z^{\prime}\in{\PP^{2[d_{r+1}]}}|\exists Z\in{\mathscr{A}} \text{ such that } Z^{\prime}=X_{2(r+1)}\cap Y_{2(r+1)}-Z\}\subseteq \PP^{2[d_{r+1}]} $$
    with $X_{2(r+1)}$ and $Y_{2(r+1)}$ smooth curves of degree $2(r+1)$.
\end{definition}

\begin{prop} \label{prop1}
    For all $r\geq 1$ is satisfied:
    $$\mathscr{L}D_{r}\subseteq D_{r+1}$$
\end{prop}
\begin{proof}
    Let $Z\in{\mathscr{L}D_{r}}$, by definition there exists $Z^{\prime}\in{D_{r}}$ residual to $Z$ by the complete intersection of two curves $X=V(F)$ and $Y=V(G)$ of degrees $2(r+1)$. Since that $Z^{\prime}$ is an element of $D_{r}$, by the remark \ref{obs1}, $Z^{\prime}$ have minimal free resolution as (\ref{restangext}), following the notation of \cite[pp.280]{PS} we have that $Z$ has resolution:
\begin{equation} \label{tanr+1}
\xymatrix{0\to \sum_{1}^{m} \mathcal{O}_{\mathds{P}^{2}}(-N_{j})\ar[r]^{\psi}& \sum_{1}^{m+1}\mathcal{O}_{\mathds{P}^{2}}(-D_{i})\ar[r]& \mathscr{I}_{Z}\to 0}
\end{equation}

with $m=r+2$, $N_{j}=2(2(r+1))-d_{j}$ for $j\in{\{1,\ldots , m\}}$, $D_{i}=2(2(r+1))-n_{i}$ for $i\in{\{1,\ldots , m-1\}}$ and $D_{m}=D_{m+1}=2(r+1)$, replacing the values in the resolution(\ref{tanr+1}) we obtain the resolution
{\footnotesize$$\xymatrix{   0 \to \mathcal{O}_{\PP^{2}} (-2(r+1)-1)  \oplus \mathcal{O}_{\PP^{2}}(-2(r+1)-2)^{r+1}\ar[r]^{\quad\psi}& \mathcal{O}_{\PP^{2}}(-2(r+1))^{(r+1)+1} \oplus \mathcal{O}_{\PP^{2}}(-2(r+1)-1)  \ar[r]& \mathscr{I}_{Z} \to 0}$$}

On the other hand, $F=\displaystyle \sum_{1}^{m} \lambda_{i}A_{i}$ and $G=\displaystyle \sum_{i}^{m}\mu_{i}A_{j}$ with $A_{i}$ the $(m-1)$-minors of the matrix $[\varphi]$, thus $\lambda_{1}$ and $\mu_{1}$ are homogeneous polynomials of degree $1$ and $\lambda_{i}$,$\mu_{i}$ with $i\in{\{2,\ldots , m\}}$ are homogeneous polynomials of degree $2$, then the matrix:
$$[\psi]=\left (
\begin{array}{c}
\begin{array}{|c|} \hline \\
\quad ^{t}[\varphi] \quad\\
\\ \hline
\end{array} \\
\lambda_{1}  \ldots  \lambda_{m} \\
      \mu_{1} \ldots  \mu_{m}
\end{array}
\right )$$
do not have constant entries no zero, thus the resolution is minimal, and again by the remark \ref{obs1} we have that $Z$ is an element of $D_{r+1}$.\\
\end{proof}

\begin{rmk}
    Given a general point $Z$ on $D_{r}$ and $Z^{\prime}$ on $\mathscr{L}D_{r}\subseteq D_{r+1}$, from its resolution (\ref{restangext}) we have that $$h^{0}(\PP^{2},\mathscr{I}_{Z}(2(r+1))=5r+6 \text{   and   } h^{0}(\PP^{2},\mathscr{I}_{Z^{\prime}}(2(r+1))=r+2$$
\end{rmk}

\begin{cor}
    For all $r\geq 1$ we have that $$\mathscr{L}D_{r}= D_{r+1}$$
\end{cor}
\begin{proof}
Since $D_{r+1}$ is irreducible and of dimension $2d_{r+1}-1$ and by the Proposition \ref{prop1} we have that $\mathscr{L}D_{r}\subseteq D_{r+1}$, thus it is enough to prove that the dimension of $ \mathscr{L}D_{r}$ is $2d_{r+1}-1$.
Let $W_{r}$ the incidence variety defined by: $$ \xymatrix{&&&W_{r}:=\{ (Z,Z^{\prime})\in{D_{r}\times \mathscr{L}D_{r}} | Z\cup Z^{\prime}=X_{2(r+1)}\cap Y_{2(r+1)} \} \ar[ld]_{\pi_{r}^{0}}\ar[rd]^{\pi_{r}^{1}}&&&\\
&&D_{r} & &\mathscr{L}D_{r}&&}$$
with $X_{2(r+1)}$ and $Y_{2(r+1)}$ curves of degree $2(r+1)$. Then we have that:
\begin{align*}
    dim\,\mathscr{L}D_{r} &= dim\,W_{r} - dim\,(\pi_{r}^{1})^{-1}\\
    & = (dim\, D_{r} + dim\,(\pi_{r}^{0})^{-1})- dim\,(\pi_{r}^{1})^{-1}\\
    &= (2d_{r}-1+dim\,G(2,h^{0}(\PPP, \ii_{Z}(2(r+1))) )) -dim\,G(2, h^{0}(\PPP, \ii_{Z^{\prime}}(2(r+1))) ) \\
    &= 2d_{r}-1+2(5r+6-2)-2(r+2-2)\\
    &= 2d_{r+1}-1
\end{align*}
\end{proof}

\end{document}